\documentclass[12pt,a4paper]{article}
\usepackage{amsmath,amssymb,centernot, amsthm}
\usepackage{showlabels}

\newtheorem{thm}{Theorem}[section]
\newtheorem{lem}[thm]{Lemma}

\newtheorem{cor}[thm]{Corollary}

\newtheorem{re}[thm]{Result}
\newtheorem{question}[thm]{Question}

\newtheorem{example}[thm]{Example}

\newtheorem{remark}[thm]{Remark}

\newcommand{\Z}{\mathbb{Z}}

\newcommand{\C}{\mathbb{C}}
\newcommand{\Q}{\mathbb{Q}}

\DeclareMathOperator{\BH}{ {\rm BH} }

\begin{document}
\title{Bilinear Forms on Finite Abelian Groups and Group-Invariant
 Butson Hadamard Matrices}
\author{Tai Do Duc\\ Division of Mathematical Sciences\\
School of Physical \& Mathematical Sciences\\
Nanyang Technological University\\
Singapore 637371\\
Republic of Singapore\\[5mm]
Bernhard Schmidt\\ Division of Mathematical Sciences\\
School of Physical \& Mathematical Sciences\\
Nanyang Technological University\\
Singapore 637371\\
Republic of Singapore}

\maketitle

 \begin{abstract}
Let $K$ be a finite abelian group
and let $\exp(K)$ denote the least common multiple of the orders of the elements of $K$.
A $\BH(K,h)$  matrix is a $K$-invariant $|K|\times |K|$ matrix $H$ whose entries are complex $h$th
roots of unity such that $HH^*=|K|I$, where $H^*$ denotes the complex conjugate
transpose of $H$, and  $I$ is the identity matrix of order $|K|$.
Let $\nu_p(x)$ denote the $p$-adic valuation of the integer $x$.
Using bilinear forms on $K$, we  show that a $\BH(K,h)$ exists whenever
\begin{itemize}
\item[(i)] $\nu_p(h) \geq \lceil \nu_p(\exp(K))/2 \rceil$ for every prime divisor $p$ of $|K|$ and
\item[(ii)] $\nu_2(h) \ge 2$  if  $\nu_2(|K|)$ is odd and $K$ has a direct factor $\Z_2$.
\end{itemize}
Employing the field descent method, we prove that these conditions are necessary for the existence of  a
$\BH(K,h)$ matrix in the case where $K$ is cyclic of prime power order.
\end{abstract}


\section{Introduction}
Let $U_h$ be the set of complex $h$th roots of unity.
An $n\times n$-matrix $H$ with entries from $U_h$ is called a 
\textbf{Butson Hadamard matrix} if $HH^* = nI$,
where $H^*$ is the complex conjugate transpose of $H$ and $I$ is the identity matrix of order $n$.
We say that $H$ is a {\boldmath$\BH(n,h)$} matrix.
The Ph.D.\ thesis \cite{szoe} of Sz\"{o}ll\H{o}si provides a good overview of most of the known results
on Butson Hadamard matrices and \cite{lam, hir} contain more recent work and survey open problems in this area.

\medskip

The focus of this paper is Butson Hadamard matrices invariant under abelian groups.
Let $(G,+)$ be a finite abelian group of order $n$ with identity element $0$.
An $n\times n$ matrix $A=(a_{g,k})_{g,k\in G}$ is
{\boldmath $G$}\textbf{-invariant} if $a_{g+l,k+l}=a_{g,k}$ for all $g,k,l\in G$.
Such a matrix is sometimes also called \textbf{group invariant} or \textbf{group developed}. 
A $G$-invariant $\BH(n,h)$ matrix is also called a {\boldmath$\BH(G,h)$} \textbf{matrix}.

\medskip

\begin{remark} \label{mult}
For any multiple $h'$ of $h$,
every $\BH(G,h)$ matrix is also a $\BH(G,h')$ matrix, as $U_h\subset U_{h'}$.
\end{remark}

\medskip

By $\Z_n$ we denote the cyclic group of order $n$.
Most existing work on group invariant Butson Hadamard matrices concerns circulant matrices, i.e.,
$\BH(\Z_n,h)$ matrices.
Backelin \cite{bac} came up with the following result.
\begin{re} \label{back}
Let $p$ be a prime and let $n$ be a positive integer such that $n\equiv 0\ (\bmod\ p^2)$ and
$n\not\equiv 2\ (\bmod\ 4)$. Then there is a $\BH(\Z_n,n/p)$ matrix.
\end{re}
We remark that Backelin actually formulated his result in terms of so called
cyclic $n$-roots; Result \ref{back} translates his theorem into the language
of group invariant Butson Hadamard matrices. We further remark that the condition
$n\not\equiv 2\ (\bmod\ 4)$ is missing in the statement of the theorem
in Backelin's paper, but is necessary for his construction to work.
In fact, for instance, it can be shown \cite{leu} that $\BH(\Z_{2p^2},2p)$ matrices do not exist
for any odd prime $p$.
The special case $n=p^2$ of Result \ref{back} was rediscovered in \cite{del}.

\medskip

The main purpose of this paper is to provide a vast generalization and strengthening
of Result \ref{back}. In fact, it turns out that any nondegenerate symmetric bilinear form
on a finite abelian group can be used to construct group invariant Butson Hadamard matrices.
Within our construction, given an abelian group $G$, there is ample freedom to choose ``ingredients''
(the bilinear form, a suitable subgroup of $G$, and a system of coset representatives).

\medskip

There is a well developed theory of bilinear forms on finite abelian groups,
see \cite{mir,wal,wal1}. It is shown in these papers that,
contrary to the case of bilinear forms over finite fields,
in general there is quite a number of inequivalent nondegenerate symmetric bilinear
forms on a finite abelian group.
As \textit{any} of these bilinear forms  can be used
in our construction, this theory turns out to be relevant for
the existence of Butson Hadamard matrices and sheds light on the
above mentioned flexibility of ingredients. We will not discuss
the theory of bilinear forms on finite abelian groups in this paper though,
and just focus on proving the correctness of our construction.

\medskip

Finally, we will show that the conditions which are sufficient
for our construction of $\BH(G,h)$ matrices to work are also necessary
for the existence of these matrices in the case
where $G$ is cyclic of prime power order. 
The proof is an application of the field descent method developed in \cite{sc2}.
We remark that the field descent method relies on the fact that algebraic 
integers in a cyclotomic field $F$ whose squared modulus is an integer
often are contained in proper subfields of $F$, which are, in fact,
are also cyclotomic fields. This method was introduced in
\cite{sc2} and has, for instance, been used to obtain
progress on the Circulant Hadamard Matrix Conjecture \cite{sc2,sc1}
and Lander's Conjecture \cite{leu1}.


\section{The Construction}
For a finite abelian group $G$, we denote the least common multiple of the orders
of the elements of $G$ by $\exp(G)$.
For a positive integer $t$, write $\zeta_t = \exp(2\pi i/t)$.
As before, we denote the cyclic group of order $t$ by $\Z_t$ and identify $\Z_t$
with $\{0,\ldots,t-1\}$, the group operation being addition modulo $t$.

\medskip

Let $G$ be a finite abelian group and let $e$ be a positive integer. We say that
a map $f: G\times G \to \Z_e$ is a \textbf{bilinear form} if
\begin{equation} \label{bilin}
\begin{split}
f(g+h,k) & =f(g,k)+f(h,k) \text{ and }\\
f(g,h+k) &=f(g,h)+f(g,k)
\end{split}
\end{equation}
for all $g,h,k\in G$.
Note that (\ref{bilin}) implies
\begin{equation*}
\begin{split}
f(\alpha g,k) & = \alpha f(g,h) \text{ and }\\
f(g,\alpha h) &=\alpha f(g,h)
\end{split}
\end{equation*}
for all $g,h\in G$ and $\alpha\in \Z$.
If $f(g,h)=f(h,g)$ for all $g,h\in G$, then $f$ is \textbf{symmetric}.
If $f(g,h)=0$, then $g$ and $h$ are said to be \textbf{orthogonal}.
We say that $f$ is \textbf{nondegenerate}
if there is no $g\in G\setminus\{0\}$ such that $f(g,h)=0$ for all $h\in G$.

\medskip

We will use the following conventions. For an abelian group $G$
and $g\in G$, we say that $h\in G$ is a \textbf{square root} of $g$ if $g=2h$
and we write $h= g/2$.
Note that square roots are not unique in general; 
for our purposes, $g/2$ denotes any square root of $g$.
In fact, the construction in Theorem \ref{blc} works
no matter which square roots are chosen.

\medskip

Let $L$ be an elementary abelian group of order $2^{2a+b}$
where $b\in \{0,1\}$ and write $c=2a+b$.
We identify $L$ with $\{(g_1,\ldots,g_c): g_1,\ldots,g_{c}\in \{0,1\}\}$,
 the group operation being componentwise addition modulo $2$.
For $g=(g_1,\ldots,g_c)\in L$,
set $G_1=(g_1,\ldots,g_a)$ and $G_2=(g_{a+1},\ldots,g_{2a})$.
Similarly, define $X_1$ and $X_2$ for $x=(x_1,\ldots,x_c)\in L$.
For $x,y\in \{0,1\}^a$, write $xy^T = \sum_{i=1}^a x_iy_i$.
Note that in this sum we use the addition of integers, not addition modulo $2$.
For instance, if $x=y=(1,1)$, then $xy^T$ is the integer $2$.
Define a function
$s_L: L\to \Z$ by
$$
s_L(g)=
\left\{
\begin{array}{ll}
2 G_1G_2^T & \text{ if } b=0,\\
 2 G_1G_2^T + g_c & \text{ if } b=1.
\end{array}
\right.
$$
For $u,w\in \Z$, set $u\oplus w=0$ if $u+w$ is even and $u\oplus w=1$ if $u+w$ is odd.
For $x,y\in L$, define $x\oplus y$ by $(x\oplus y)_i = x_i\oplus y_i$ for all $i$.
First assume $b=1$. Let $g$ be any element of $L$.
We have
\begin{equation} \label{sl1}
\begin{split}
\sum_{x\in L} \zeta_4^{s_L(x)-s_L(x\oplus g)} &
    = \sum_{x\in L} \zeta_4^{2X_1X_2^T + x_c-2(X_1X_2^T + X_1G_2^T + G_1X_2^T + G_1G_2^T) - (x\oplus g)_c}\\
    & = (-1)^{G_1G_2^T} \sum_{x_c=0}^1 \zeta_4^{x_c-(x_c\oplus g_c)} 
      \sum_{X_1\in \{0,1\}^a}(-1)^{X_1G_2^T} \\
      & \mbox{\hspace{1,9cm}} \sum_{X_2\in \{0,1\}^a}(-1)^{G_1X_2^T}.
\end{split}
\end{equation}
If $G_1\neq (0,\ldots,0)$, then the last sum in (\ref{sl1}) vanishes and
if $G_2\neq (0,\ldots,0)$, then the second last sum in (\ref{sl1}) vanishes.
If $g_c\neq 0$, then $g_c=1$ and
$$
\sum_{x_c=0}^1 \zeta_4^{x_c-(x_c\oplus g_c)} = \zeta_4^{-1} + \zeta_4^{1} =0.
$$
In summary, we have
\begin{equation} \label{sl}
\sum_{x\in L} \zeta_4^{s_L(x)-s_L(x\oplus g)}=0
\end{equation}
whenever $g\neq 0$. If $b=0$, then (\ref{sl}) also holds and is proved in a similar way.

\medskip

\begin{thm} \label{blc}
Let $K=G\times L$ be a finite abelian group,
where either $L=\{0\}$ or $L$ is an elementary abelian $2$-group. Write $e=\exp(G)$.
Let $U$ be a subgroup of $G$ such that every element of $U$ has a square root in $G$.
Suppose that $f: G\times G \to \Z_e$ is bilinear, symmetric, and nondegenerate,
and  that no element of $G\setminus U$ is orthogonal to all elements of $U$.
Let $R\subset G$ be a complete
system of coset representatives of $U$ in $G$ with $0\in R$.
For every $x\in K$, there are unique $x_1\in U$, $x_2\in R$, and $x_3\in L$ with $x = x_1+x_2+x_3$.
Let $\beta$ be any integer coprime to $|G|$.
Define a matrix $H=(H_{y,x})_{y,x\in K}$ by
\begin{equation} \label{blc1}
H_{y,x} =
\zeta_e^{f((x-y)_1/2, (x-y)_1) + \beta f((x-y)_1,(x-y)_2)} \zeta_4^{s_L(x_3\oplus y_3)}.
\end{equation}
Then $H$ is a  $\BH(K,e_1)$ matrix, where
$$e_1=
\left\{
\begin{array}{ll}
\exp(U) &  \text{ if } L=\{0\},\\
{\rm lcm}(2,\exp(U)) & \text{ if $L$ is of square order},\\
{\rm lcm}(4,\exp(U)) & \text{ otherwise}.
\end{array}
\right.
$$
\end{thm}

\begin{proof}
We first make a remark on the assumption that
every element of $U$ has a square root in $G$, as it will not be mentioned
again in the proof. This assumption is necessary for 
the right hand side of (\ref{blc1}) to be properly defined (note that 
$(x-y)_1/2$ must exist for all $x,y\in K$). 

\medskip

From the definition, it is clear  that $H$ is $G$-invariant.
Fix any $y\in K$.
For every $x\in K$, there is a unique $u(x_2)\in U$ with
\begin{equation} \label{blc2a}
(x-y)_2 = x_2 -y_2 +u(x_2).
\end{equation}
Note that $u(x_2)$ depends on $y$, but we do not indicate this dependence, as we consider $y$ as fixed.
We have
\begin{equation} \label{blc2b}
(x-y)_1 = x_1 -y_1 - u(x_2),
\end{equation}
by (\ref{blc2a}), as $x_1+x_2 - y_1 -y_2 = (x-y)_1 + (x-y)_2$.
Note that $u(x_2)$ only depends on $x_2$ (not on $x_1$), as $(x-y)_2 = (x'-y)_2$ whenever $x_2=x'_2$.
We claim that
\begin{equation} \label{blc3}
\begin{split}
u(x_2) = 0 \text{ whenever } y_2=0.
\end{split}
\end{equation}
Indeed, if $y_2=0$, then $u(x_2) = (x-y)_2 - x_2 = (x_1+x_2+x_3-y_1-y_3)_2 - x_2 = x_2-x_2=0$.

\medskip

By (\ref{blc1}), (\ref{blc2a}), and (\ref{blc2b}), we have
\begin{equation} \label{blc3a}
H_{y,x} =
\zeta_e^{f( (x_1-y_1-u(x_2))/2, x_1-y_1-u(x_2)) + \beta f(x_1-y_1-u(x_2),x_2-y_2+u(x_2))} \zeta_4^{s_L(x_3\oplus y_3)}.
\end{equation}
Note that the $(0,y)$ entry of $HH^*$ is
\begin{equation*}
\begin{split}
A(y): & = \sum_{x \in K}H_{0,x}\overline{H_{y,x}}\\
& = \sum_{x_3\in L} \sum_{x_2\in R}\sum_{x_1\in U} H_{0,x_1+x_2+x_3}\overline{H_{y,x_1+x_2+x_3}}.
\end{split}
\end{equation*}
Substituting (\ref{blc3a}) into the last expression and 
using the bilinearity of $f$, we get
\begin{equation*}
\begin{split}
A(y) & = \eta \sum_{x_3\in L} \zeta_4^{s_L(x_3)-s_L(x_3\oplus y_3)}
     \sum_{x_2\in R} \zeta_e^{T(x_2)} \sum_{x_1\in U} \zeta_e^{S(x_1,x_2)},
\end{split}
\end{equation*}
where
\begin{eqnarray*}
\eta & = & \zeta_e^{-f(y_1/2,y_1)-\beta f(y_1,y_2)},\\
T(x_2) & = &     \beta f(x_2,y_1+u(x_2)) \\
                && + f(u(x_2),-u(x_2)/2 -y_1 +\beta(u(x_2)+y_1-y_2)), \\
S(x_1,x_2)  & = & f(x_1,y_1+\beta y_2+ (1- \beta ) u(x_2)).
\end{eqnarray*}

\medskip

From now on, we assume $A(y)\neq 0$. Our goal is to show that this implies $y=0$.
First of all, note that $A(y)\neq 0$ implies
$
\sum_{x_3\in L} \zeta_4^{s_L(x_3)-s_L(x_3\oplus y_3)} \neq 0
$
and thus $y_3=0$ by (\ref{sl}).
Hence the first sum in the expression for $A(y)$ is equal to $|L|$ and
\begin{equation} \label{blc4}
A(y) = \eta |L|
     \sum_{x_2\in R} \zeta_e^{T(x_2)} \sum_{x_1\in U} \zeta_e^{S(x_1,x_2)}.
\end{equation}

\medskip

Write $V(x_2)=\sum_{x_1\in U} \zeta_e^{S(x_1,x_2)}$.
Note that $S(x_1+z_1,x_2)=S(x_1,x_2)+S(z_1,x_2)$ for all $x_1,z_1\in U$, since
$f$ is bilinear. Suppose that there is $z_1\in U$ such that
$S(z_1,x_2)\neq 0$ and thus $\zeta_e^{S(z_1,x_2)}\neq 1$.
As
$$
\zeta_e^{S(z_1,x_2)} V(x_2) = \sum_{x_1\in U} \zeta_e^{S(x_1+z_1,x_2)}
= \sum_{x_1\in U} \zeta_e^{S(x_1,x_2)} = V(x_2),
$$
we conclude $V(x_2)=0$. So we see that $V(x_2)$ is only nonzero
if
\begin{equation} \label{blc5}
 S(x_1,x_2) = f(x_1,y_1+\beta y_2+ (1- \beta ) u(x_2))=0
\end{equation}
for all $x_1\in U$.
As no element of $G\setminus U$ is orthogonal to all elements of $U$ by assumption,
this implies $y_1+\beta y_2+ (1- \beta ) u(x_2)\in U$ and thus $y_2\in U$, as $y_1, u(x_2)\in U$ and $\beta$ is coprime to $|G|$.
This implies $y_2=0$, as $y_2\in R\cap U$ and $R\cap U = \{0\}$.
 Hence we have $y_2=0$ whenever $V(x_2)\neq 0$.
As $A(y)\neq 0$ by assumption, we have $V(x_2)\neq 0$ for at least one $x_2$ and thus
\begin{equation} \label{blc6}
y_2=0.
\end{equation}
Hence
\begin{equation} \label{blc6a}
u(x_2)=0
\end{equation}
for all $x_2\in R$  by (\ref{blc3}). Combining (\ref{blc5}),  (\ref{blc6}), and (\ref{blc6a}), we get
\begin{equation} \label{blc7}
f(x_1,y_1)=0  \text{ for all } x_1\in U
\end{equation}
and
\begin{equation} \label{blc7a}
A(y) = \eta |L| |U|  \sum_{x_2\in R} \zeta_e^{T(x_2)}
     =  \eta |L| |U|  \sum_{x_2\in R} \zeta_e^{\beta f(y_1,x_2)}.
\end{equation}
Hence $W(y):=\sum_{x_2\in R} \zeta_e^{\beta f(y_1,x_2)}\neq 0$.

\medskip

Now suppose that there is $z_2\in R$ such that $f(y_1,z_2)\neq 0$.
Note that the map $R\to R,\ x_2\mapsto (x_2+z_2)_2$ is a bijection,
as the elements of $R$ represent each coset of $U$ in $G$ exactly once.
Moreover, for any $r,s\in R$, we have
\begin{equation*}
\begin{split}
 f(y_1,r+s) &=  f( y_1, (r+s)_1 + (r+s)_2) \\
& = f(y_1, (r+s)_2),
\end{split}
\end{equation*}
since $f(y_1,(r+s)_1)=0$ by (\ref{blc7}). Thus
\begin{equation*}
\begin{split}
\zeta^{\beta f(y_1,z_2)} W(y) & =  \sum_{x_2\in R} \zeta_e^{\beta f(y_1,x_2+z_2)}\\
& = \sum_{x_2\in R} \zeta_e^{\beta f(y_1,(x_2+z_2)_2)}\\
& = \sum_{x_2\in R} \zeta_e^{\beta f(y_1,x_2)} = W(y).
\end{split}
\end{equation*}
As $\zeta_e^{\beta f(y_1,z_2)}\neq 1$, we conclude $W(y)=0$, a contradiction.
Hence we have $f(y_1,x_2)=0$ for all $x_2\in R$.
This, together with (\ref{blc7}), implies $f(y_1,x)=f(y_1,x_1+x_2) = f(y_1,x_1)+f(y_1,x_2)=0$ for all $x\in G$.
Thus $y_1=0$, as $f$ is nondegenerate.
So we have  $y_3=0$, $y_2=0$, and $y_1=0$, that is, $y=0$, as desired.

\medskip

In summary, we have shown that the $(0,y)$ entry of $HH^*$ is only nonzero if $y=0$.
As $H$ is $K$-invariant, this shows that the $(x,y)$ entry of $HH^*$ is only nonzero
if $x=y$. Hence $HH^* = |K|I$, where $I$ is the identity matrix of order $|K|$.

\medskip

It remains to prove that the entries of $H$ are $e_1$th roots of unity,
where $e_1$ is defined in the statement of the theorem.
Recall that
\begin{equation} \label{blc9a}
H_{y,x} = W_{y,x} \zeta_4^{s_L(x_3\oplus y_3)},
\end{equation}
where $W_{y,x}= \zeta_e^{f((x-y)_1/2, (x-y)_1) + f((x-y)_1,(x-y)_2)}$.
Write $k=\exp(U)$.
As $kx=0$ for all $x\in U$, we have $kf(x,y)=f(kx,y)=f(y,kx)=0$ for all $x\in U$ and $y\in G$
by the bilinearity of $f$.
As $(x-y)_1\in U$ for all $x,y\in G$, this shows that
\begin{equation} \label{blc9}
W_{y,x}^k = \zeta_e^{f((x-y)_1/2, k(x-y)_1) + f(k(x-y)_1,(x-y)_2)} =\zeta_e^0=1
\end{equation}
for all $x,y\in K$.
If $L=\{0\}$, then $e_1=k$ and $H_{y,x}=W_{y,x}$ and thus $H_{y,x}^{e_1}=1$ for all $x,y\in K$
by (\ref{blc9}).
If $L$ is of square order, then $e_1={\rm lcm}(2,k)$,  $s_L(x_3\oplus y_3)\equiv 0\ (\bmod\ 2)$
and thus $H_{y,x}^{e_1} =1$ for all $x,y\in K$ by (\ref{blc9a}) and (\ref{blc9}).
Note that, in any case,  $H_{y,x}^{{\rm lcm}(4,k)} =1$ for all $x,y\in K$ by (\ref{blc9a}) and (\ref{blc9}).
Hence $H$ indeed is a $\BH(G,e_1)$ matrix.
\end{proof}

\medskip

\begin{remark} \rm
In Theorem \ref{blc}, there are alternative choices for the function $s_L$ for which the
construction still works. Write $|L|=2^{2a+d}$ and $c=2a+d$, 
where $a$ and $d$ are any nonnegative integers.
Recall that $F: (\Z_2)^{2a}\to \Z_2$ is a \textbf{bent function} if
$$
\left|\sum_{x\in (\Z_2)^{2a}}(-1)^{F(x)+ \alpha x^T}\right| = 2^a
$$
for all nonzero $\alpha \in (\Z_2)^{2a}$.
Bent functions exist in abundance, see \cite{ro1,ro2}, for instance.
Let $F: (\Z_2)^{2a}\to \Z_2$ be any bent function and set
\begin{equation} \label{general}
s_L(g_1,\ldots,g_c) = 2F(g_1,\ldots,g_{2a}) + \sum_{i=2a+1}^c g_i.
\end{equation}
Then (\ref{blc1}), with $s_L$ defined by (\ref{general}), still is a $\BH(K,e_1)$ matrix.
We omit the proof here, which is a straightforward extension of the proof of Theorem \ref{blc}.
\end{remark}

\medskip

For a prime $p$ and an integer $t$,
let $\nu_p(t)$ denote the $p$-adic valuation of $t$, that is, $p^{\nu_p(t)}$ is the
largest power of $p$ dividing $t$.
For  groups $K$ and $W$, we say that $K$ has a \textbf{direct factor}
{\boldmath $W$} if $K \cong W\times V$ for some group $V$.

\medskip

\begin{cor} \label{main_constr}
Let $K$ be a finite abelian group and let $h$ be a positive integer such that
\begin{eqnarray}
\label{main_constr1}
\text{$\nu_p(h) \geq \lceil \nu_p(\exp(K))/2 \rceil$ for every prime divisor $p$ of $|K|$,}\\
\label{main_constr1a}
\nu_2(h) \ge 2 \text{ if  $\nu_2(|K|)$ is odd and $K$ has a direct factor $\Z_2$}.
\end{eqnarray}
Then there exists a $\BH(K,h)$ matrix.
\end{cor}
\begin{proof}
Write $K=G\times L$, where either $L=\{0\}$ or $L$ is an elementary abelian $2$-group,
such that $G$ does not have a direct factor $\Z_2$.
Without loss of generality, we can assume
$G = \Z_{p_1^{a_1}} \times \cdots \times \Z_{p_s^{a_s}}$, where $p_1,\ldots,p_s$ are
(not necessarily distinct) primes and $a_1,\ldots,a_s$ are positive integers.
Note that $p_i^{a_i}\neq 2$ for all $i$, as $G$ has no direct factor $\Z_2$.
Furthermore, $|L|>1$ if and only if $K$ has a direct factor $\Z_2$.
We identify $G$ with
$$
 \{(g_1,\ldots,g_s): 0\le g_i \le p_i^{a_i}-1,\ i=1,\ldots,s\},
$$
where the group operation is componentwise addition, the $i$th component being taken
modulo $p_i^{a_i}$. Let $U$ be the subgroup of $G$ given by
$$
U = \left\{\left(p_1^{\lfloor a_1/2 \rfloor}k_1,\ldots, p_s^{\lfloor a_s/2 \rfloor}k_s\right):
       0\le k_i \le p_i^{\lceil a_i/2 \rceil}-1,\ i=1,\ldots,s\right\}.
$$
Write $e=\exp(G)$ and define a map $f:G\times G\to \Z_e$ by
$$
f((g_1,\ldots,g_s), (h_1,\ldots,h_s)) = \sum_{i=1}^s \frac{e}{p_i^{a_i}} g_ih_i .
$$
It is straightforward to verify that $f$ is bilinear and symmetric.
Let $g=(g_1,\ldots,g_s)$ be any nonzero element of $G$.
Then $g_i\neq 0$ for some $i$. Let $t$ be the element of $G$ with $t_i=1$
and $t_j=0$ for all $j\neq i$.
Then
$f(g,t) = \frac{e}{p_i^{a_i}} g_i \neq 0$, since $g_i\not\equiv 0\ (\bmod\ p_i^{a_i})$.
This shows that $f$ is nondegenerate.

\medskip

Suppose that $g=(g_1,\ldots,g_s)$ an element which orthogonal to all elements of $U$.
Let $i$ be arbitrary and let $s$ be the element of $U$ with $s_i=p_i^{\lfloor a_i/2 \rfloor}$ and $s_j=0$
for all $j\neq i$. Then
$$
f(g,s) =  \frac{e}{p_i^{a_i}} g_i p_i^{\lfloor a_i/2 \rfloor} = \frac{e}{ p_i^{\lceil a_i/2 \rceil}}g_i=0
$$
and thus $g_i \equiv 0\ (\bmod\ p_i^{\lceil a_i/2 \rceil})$. This implies $g\in U$.
Hence there is no element of $G\setminus U$ which is orthogonal to all elements of $U$.

\medskip

Next, we show that every element $u=(u_1,\ldots,u_s)$
of $U$ has a square root in $G$.
If $p_i=2$, then $a_i\ge 2$, as we are assuming $p_i^{a_i}\neq 2$ for all $i$.
Thus $u_i\equiv 0\ (\bmod\ 2)$ by the definition of $U$, that is, $u_i=2y_i$
with $y_i\in \{0,\ldots,2^{a_i-1}-1\}$.
On the other hand, if $p_i$ is odd, then the map $x\mapsto 2x\ (\bmod\ p_i^{a_i})$
is a bijection and thus there is $y_i\in \{0,\ldots,p^{a_i}-1\}$ with
$2y_i \equiv u_i\ (\bmod\ p_i^{a_i})$.
In summary, we have $u=2y$, i.e., $y$ is a square root of $u$.

\medskip

We have shown that all assumptions of Theorem \ref{blc} are satisfied.
Hence there exists a $\BH(K,e_1)$ matrix, where $e_1$
is defined in Theorem \ref{blc}.
In view of Remark \ref{mult}, it suffices to show
\begin{equation} \label{main_constr3}
h\equiv 0\ (\bmod\ e_1).
\end{equation}
Note that
\begin{equation} \label{main_constr2}
\exp(U) = {\rm lcm}\left(p_i^{\lceil a_i/2 \rceil}: i=1,\ldots,s\right)
\end{equation}
and $h\equiv 0\ \left(\bmod\ p_i^{\lceil a_i/2 \rceil}\right)$ by (\ref{main_constr1}).
Hence
\begin{equation} \label{main_constr4}
h\equiv 0\ (\bmod\ \exp(U)).
\end{equation}
If $L=\{0\}$, then $e_1=\exp(U)$ and (\ref{main_constr3}) follows from (\ref{main_constr4}).
Now assume $|L|>1$. If $\nu_2(\exp(K))\ge 3$, then $h \equiv 0\ (\bmod\ 4)$ by
(\ref{main_constr1}) and thus (\ref{main_constr3}) holds by (\ref{main_constr4}),
as $e_1$ divides ${\rm lcm}(4,\exp(U))$.
Hence we can assume $\nu_2(\exp(K))\le 2$, i.e., $a_i=2$ whenever $p_i=2$.
Note that this implies that $\nu_2(|G|)$ is even.
Suppose that $\nu_2(|K|)$ is even.
Then $\nu_2(|L|)$ is also even, that is, $L$ is of square order.
Thus $e_1={\rm lcm}(2,\exp(U))$ and (\ref{main_constr3})
follows from (\ref{main_constr1}) and (\ref{main_constr4}).
Finally, suppose that $\nu_2(|K|)$ is odd.
Note that $K$ has a direct factor $\Z_2$, as we are assuming $|L|>1$.
Thus $h\equiv 0\ (\bmod\ 4)$ by (\ref{main_constr1a}), which implies (\ref{main_constr3}).
Thus (\ref{main_constr3}) holds in every case and this completes the proof.
\end{proof}

\medskip

\begin{remark} \rm
Corollary \ref{main_constr} only provides one possible choice of the bilinear form $f$
and the subgroup $U$. In general, there are numerous other choices,
which produce group invariant Butson Hadamard matrices not equivalent to those
constructed in Corollary \ref{main_constr}.
\end{remark}

\medskip

\begin{cor} \label{circ}
If $v$ and $h$ are positive integers with
\begin{itemize}
\item[$(i)$]  $\nu_p(h) \geq \lceil \nu_p(v)/2 \rceil$ for every prime divisor $p$ of $v$ and
\item[$(ii)$] $\nu_2(h) \ge 2$  if  $v\equiv 2\ (\bmod\ 4)$,
\end{itemize}
then a (circulant) $\BH(\Z_v,h)$ matrix exists.
\end{cor}


\section{Necessary Conditions}
From now on, we use the language of group ring equations to study group invariant Butson Hadamard matrices
and write groups multiplicatively.
Let $G$ be a finite abelian group, let $R$ be a ring and let
$R[G]$ denote the group ring of $G$ over $R$. The elements of $R[G]$ have
the form $X=\sum_{g \in G} a_gg$ with $a_g \in R$.
The $a_g$'s are called the \textbf{coefficients} of $X$.
Two elements $X=\sum_{g \in G} a_gg$ and $Y=\sum_{g\in G}b_gg$ in $R[G]$
are equal if and only if $a_g=b_g$ for all $g\in G$.
A subset $S$ of $G$ is identified with the group ring element $\sum_{g\in S}g$.
For the identity element $1_G$ of $G$ and $\lambda\in R$,  we write $\lambda$
for the group ring element $\lambda 1_G$.
For $R=\mathbb{Z}[\zeta_h]$ and $X=\sum_{g \in G}a_gg\in R[G]$,
we write
$$
X^{(-1)}=\sum_{g\in G} \overline{a_g} g^{-1},
$$
where $\overline{a_g}$ denotes the complex conjugate of $a_g$.

\medskip

The group of complex characters of $G$ is denoted by $\hat{G}$.
The \textbf{trivial character}  $\chi_0$ is defined by $\chi_0(g)=1$ for all $g\in G$.
For $D=\sum_{g\in G} a_gg \in R[G]$ and $\chi \in \hat{G}$, write $\chi(D)=\sum_{g\in G} a_g\chi(g)$.
The following is a standard result and a proof can be found \cite[Ch.\ VI, Lem. 3.5]{bethjung}, for instance.

\begin{re} \label{fourier}
Let $G$ be a finite abelian group and $D=\sum_{g\in G} a_gg \in \mathbb{C}[G]$.
Then $$a_g=\frac{1}{|G|}\sum_{\chi \in \hat{G}} \chi(Dg^{-1})$$
for all $g\in G$.
Consequently, if $D, E \in \mathbb{C}[G]$ and $\chi(D)=\chi(E)$ for all $\chi \in \hat{G}$, then $D=E$.
\end{re}

\begin{lem} \label{div}
Let $p$ be a prime and let $m$ be a positive integer not divisible by $p$.
Write $\zeta=\zeta_p$ if $p$ is odd and $\zeta=\zeta_4$ if $p=2$.
If $X\in \Z[\zeta_m]$ satisfies $X\equiv 0\ (\bmod\ 1-\zeta)$, then $X\equiv 0\ (\bmod\ p)$.
\end{lem}

\begin{proof}
In this proof, we use basic algebraic number theory as covered in \cite{ire}, for instance.
Let $R= \Z[\zeta_{pm}]$ if $p$ is odd and $R=\Z[\zeta_{4m}]$ if $p=2$.
The ideal $(1-\zeta)R$ of $R$ factorizes as
$$(1-\zeta)R= \prod_{i=1}^k \mathfrak{p}_i,$$
where $k= \varphi(m)/{\rm ord}_m(p)$ and the $\mathfrak{p}_i$'s are distinct prime ideals.
Furthermore, $p\Z[\zeta_m] = \prod_{i=1}^k \mathfrak{q}_i$, where $\mathfrak{q}_i=\mathfrak{p}_i^{p-1}$
and the $\mathfrak{q}_i$'s are prime ideals of $\Z[\zeta_m]$.
Note that $X\equiv 0\ (\bmod\ \mathfrak{p}_i)$ for all $i$, as $X\equiv 0\ (\bmod\ 1-\zeta)$ by assumption.
Since $X\in \Z[\zeta_m]$, this implies  $X\equiv 0\ (\bmod\ \mathfrak{q}_i)$ for all $i$ and hence
$X\equiv 0\ (\bmod\ p)$.
\end{proof}

\begin{lem} \label{gr}
Let $G$ be a finite abelian group, let $h$ be a positive integer,
and let $a_g$, $g\in G$, be elements of $\{\zeta_h^i: i=0,\ldots,h-1\}$.  
Consider the element
$D = \sum_{g \in G}a_g g$ of $\Z[\zeta_h][G]$
and the $G$-invariant matrix $H=(H_{g,k})$, $g,k\in G$ given by $H_{g,k}=a_{g-k}$.
Then $H$ is a $\BH(G,h)$ matrix if and only if
\begin{equation} \label{gr1}
DD^{(-1)}=|G|.
\end{equation}
Moreover, $(\ref{gr1})$ holds if and only if
\begin{equation} \label{gr2}
|\chi(D)|^2=|G| \text{ for all } \chi\in\hat{G}.
\end{equation}
\end{lem}

\begin{proof}
Let $g\in G$ be arbitrary. The coefficient of $g$ in $DD^{(-1)}$ is
\begin{equation*}
\sum_{k,l \in G \atop k-l=g}a_{k}\overline{a_l}
= \sum_{l\in G} a_{l+g}\overline{a_{l}}.
\end{equation*}
On the other hand, the inner product of  row $x+g$ and row $x$ of $H$ is
\begin{equation*}
\sum_{k\in G} H_{x+g,k}\overline{H_{x,k}}
= \sum_{k\in G} a_{x+g-k}\overline{a_{x-k}}
 = \sum_{l\in G} a_{l+g}\overline{a_{l}}
\end{equation*}
Hence (\ref{gr1}) holds if and only if any two distinct rows of $H$ have inner product $0$,
that is, if and only if $H$ is a $\BH(G,h)$ matrix.
Finally, the equivalence of (\ref{gr1}) and (\ref{gr2}) follows from Result \ref{fourier}.
\end{proof}

\begin{re}[\cite{sc1}, Thm.\ 2.2.2] \label{fdp}
Let $p$ be a prime, let $a,b,h$ be positive integers with $(h,p)=1$, and write $v=p^ah$.
Suppose that $X$ is an element of $\Z[\zeta_{v}]$ with $|X|^2=p^b$.
Then there exist $Y\in \Z[\zeta_h]$,
a root of unity $\eta\in \Z[\zeta_h]$, $\eta\neq 1$,
of order dividing $p-1$, and an integer
$j$ such that either
$$X = \zeta_v^j Y  \text{ or } X=\zeta_v^j\Theta Y,$$
where $\Theta=1-\zeta_4$ if $p=2$ and $\Theta = \sum_{i=0}^{p-2}\eta^{-i}\zeta_p^{t^i}$ if $p$ is odd,
and $t$ is a primitive element modulo $p$.
\end{re}

\begin{thm} \label{necpp}
Let $p$ be a prime and let $a,h$ be positive integers.
If a $\BH(\Z_{p^a},h)$ matrix exists, then $p$ divides $h$.
\end{thm}

\begin{proof}
Suppose that a $\BH(\Z_{p^a},h)$ exists and that $p$ does not divide $h$.
Write $G=\Z_{p^a}$ and let $g$ be generator of $G$.
By Lemma \ref{gr}, there is $D\in \Z[\zeta_h][G]$, $D=\sum_{i=0}^{p^a-1}\zeta_h^{a_i}g^i$, with
\begin{equation} \label{necpp1}
\left|\chi(D)\right|^2 = p^a
\end{equation}
for all $\chi\in \hat{G}$. Let $\chi$ be the character of $G$ with $\chi(g)=\zeta_{p^a}$.
By Result \ref{fdp}, we have
$\chi(D) = \zeta_{p^ah}^k Y$  or $\chi(D)=\zeta_{p^ah}^k\Theta Y$
for some integer $k$,
where $Y\in \Z[\zeta_h]$ and $\Theta$ is defined in  Result \ref{fdp}.
Replacing $D$ by $\zeta_h^cg^dD$ with suitable integers $c,d$, if necessary, we can assume $k=0$.
Hence
\begin{equation} \label{necpp2a}
\chi(D) \in \Z[\zeta_h] \text{ or } \chi(D)=\Theta Y.
\end{equation}

\medskip

We first assume $p=2$. In this case, $h$ is odd and we have $\chi(D)\in \Z[\zeta_{4h}]$
by (\ref{necpp2a}). If $a=1$, it is easy to show that $h$ is divisible by $4$. Hence we can assume $a\ge 2$.
Write $D=\sum_{i=0}^{2^{a-2}-1} g^i \sum_{j=0}^3 \zeta_h^{a_{i,j}} g^{2^{a-2}j}$
with $a_{i,j}\in \Z$. We have
$$
\chi(D) = \sum_{i=0}^{2^{a-2}-1} \zeta_{2^a}^i \sum_{j=0}^3 \zeta_h^{a_{i,j}} \zeta_4^j.
$$
As  $\chi(D)\in \Z[\zeta_{4h}]$ and $\{1,\zeta_{2^a},\ldots,\zeta_{2^a}^{2^{a-2}-1}\}$ is linearly
independent over $\Q(\zeta_{4h})$, we conclude
\begin{equation*}
\begin{split}
\chi(D) & = \sum_{j=0}^3 \zeta_h^{a_{0,j}} \zeta_4^j= A + B \zeta_4,\\
\end{split}
\end{equation*}
where  $A=\zeta_h^{a_{0,0}} - \zeta_h^{a_{0,2}}$
and $B= \zeta_h^{a_{0,1}} - \zeta_h^{a_{0,3}}$.
We conclude
$$
2^a = |\chi(D)|^2 = (A+B\zeta_4)(\bar{A}-\bar{B}\zeta_4)
= A\bar{A} + B\bar{B} + (-A\bar{B} + B\bar{A})\zeta_4.
$$
As $\{1,\zeta_4\}$ is linearly independent over $\Q(\zeta_h)$,
this implies $-A\bar{B} + B\bar{A}=0$ and $A\bar{A} + B\bar{B}=2^a$.
As $h$ is odd, we have $|A|, |B| <2$ and thus $A\bar{A} + B\bar{B}<8$.
Hence $a= 2$ and $A\bar{A} + B\bar{B}=|\chi(D)|^2=4$.
A quick computation shows that this implies $\eta +\bar{\eta} + \gamma +\bar{\gamma}=0$,
where $\eta = \zeta_h^{a_{0,0}-a_{0,2}}$ and $\gamma = \zeta_h^{a_{0,1}-a_{0,3}}$.
Hence ${\rm Re}(\eta)=-{\rm Re}(\gamma)$, where ${\rm Re}(z)$ denotes the real part of $z\in \C$.
Write $\eta=\zeta_h^c$ and $\gamma = \zeta_h^d$ with $c,d\in \Z$.
We have ${\rm Re}(\eta)=\cos(2\pi i c/h)$ and ${\rm Re}(\gamma)=\cos(2\pi i d/h)$
and thus ${\rm Re}(\eta)=-{\rm Re}(\gamma)$ implies $2\pi d/h = \pm 2\pi c/h + k\pi$
where $k$ is an odd integer. This is only possible if $h$ is even.
The proof is complete for $p=2$.

\medskip
\noindent

Now assume that $p$ is odd. We rewrite $D$ in the form
$$
D = \sum_{i=0}^{p^{a-1}-1} \sum_{j=0}^{p-1} \zeta_h^{a_{i,j}} g^{i+jp^{a-1}},
$$
where the $a_{i,j}$'s are integers with $0\le a_{i,j}\le p-1$.
Note that
\begin{equation} \label{necpp2}
\chi(D) =  \sum_{i=0}^{p^{a-1}-1} \zeta_{p^a}^i \sum_{j=0}^{p-1} \zeta_h^{a_{i,j}} \zeta_p^j.
\end{equation}
As $\{1,\zeta_{p^a},\ldots,\zeta_{p^a}^{p^{a-1}-1}\}$ is linearly
independent over $\Q(\zeta_{ph})$ and  $\chi(D) \in \Z[\zeta_{ph}]$
by (\ref{necpp2a}),
we conclude that $\sum_{j=0}^{p-1} \zeta_h^{a_{i,j}} \zeta_p^j=0$ for all $i>0$ and thus
\begin{equation} \label{necpp4}
\chi(D) =   \sum_{j=0}^{p-1} \zeta_h^{b_j} \zeta_p^j,
\end{equation}
where $b_j=a_{0,j}$.
This implies $\left|\chi(D)\right|<p$ and thus $a=1$ by (\ref{necpp1}). In particular, $g$ has order $p$
and we have
\begin{equation} \label{necpp4a}
D =  \sum_{j=0}^{p-1} \zeta_h^{b_j} g^j.
\end{equation}
According to (\ref{necpp2a}), it suffices the consider the following two cases.

\medskip
\noindent
{\bf Case 1} $\chi(D)=\Theta Y$.
Recall that $\Theta = \sum_{i=0}^{p-2}\eta^{-i}\zeta_p^{t^i}$ and note that
$$
\Theta \equiv \sum_{i=0}^{p-2}\eta^{-i}\equiv 0\ (\bmod\ 1-\zeta_p),
$$
since $\eta\neq 1$ and the order of $\eta$ divides $p-1$. We conclude
$$
\sum_{i=0}^{p-1}\zeta_h^{b_i}
\equiv \sum_{i=0}^{p-1}\zeta_h^{b_i}\zeta_p^i
\equiv \chi(D) \equiv \Theta Y
\equiv 0\ (\bmod\ 1-\zeta_p).
$$
By Lemma \ref{div}, this implies
 \begin{equation} \label{necpp6}
 \sum_{j=0}^{p-1} \zeta_h^{b_j}\equiv 0\ (\bmod\ p).
 \end{equation}
However, for the trivial character $\chi_0$ of $G$,
 we have $\chi_0(D)= \sum_{j=0}^{p-1} \zeta_h^{b_j}$ by (\ref{necpp4a}) and  $|\chi_0(D)|^2=p$ by (\ref{necpp1}).
 But (\ref{necpp6}) implies $|\chi_0(D)|^2\equiv 0\ (\bmod\ p^2)$, a contradiction.

\medskip
\noindent
{\bf Case 2} $\chi(D) \in \Z[\zeta_h]$.
By (\ref{necpp4}),
we have $\chi(D)= \sum_{j=0}^{p-2}(\zeta_h^{b_j}-\zeta_h^{b_{p-1}})\zeta_p^j$
and since $\chi(D) \in \Z[\zeta_h]$ and $\{1,\zeta_p,\ldots,\zeta_p^{p-2}\}$ is linearly independent over $\Q(\zeta_h)$,
we conclude $b_1=\cdots = b_{p-1}$ and $\chi(D) = \zeta_h^{b_0}-\zeta_h^{b_{p-1}}$.
As $|\chi(D)|^2 =p$ and $p$ is odd, this implies $p=3$.
Replacing $D$ by $\zeta_h^{-b_0}D$, if necessary, we can assume $\chi(D)= 1- \zeta_h^u$
for some integer $u$. Note that $\zeta_h^u\neq 1$.
We have $3=|\chi(D)|^2 = 2 - \zeta_h^u - \zeta_h^{-u}$
and thus $\zeta_h^{u}+\zeta_h^{2u}+1=0$, which implies
$\zeta_h^{3u}=1$. As $\zeta_h^u\neq 1$, we conclude $h \equiv 0\ (\bmod\ 3)$,
contradicting the assumption that $p$ does not divide $h$.
\end{proof}

\medskip

\begin{thm} \label{necsuff}
Let $v$ be a power of a prime $p$ and let $h$ be a positive integer.
A (circulant) $\BH(\Z_v,h)$ matrix exists if and only if
\begin{equation} \label{necsuff1}
\nu_p(h) \geq \lceil \nu_p(v)/2 \rceil  \text{ and  } (v,\nu_2(h))\neq (2,1).
\end{equation}
\end{thm}

\begin{proof}
By  Corollary \ref{circ}, condition (\ref{necsuff1}) is sufficient for the existence of a $\BH(\Z_v,h)$ matrix.
It remains to prove the necessity of (\ref{necsuff1}). 
Assume that a $\BH(\Z_v,h)$ exists and write $v=p^a$,
where $a$ is a positive integer. Write $G=\Z_{p^a}$ and let $g$ be generator of $G$.
By Lemma \ref{gr}, there is $D\in \Z[\zeta_h][G]$ with $D=\sum_{i=0}^{p^a-1}\zeta_h^{a_i}g^i$ and
$DD^{(-1)}=v$.

\medskip

We first show that we can assume $a\ge 5$ if $p=2$.
Suppose $p=2$.
Then $h$ is even by Theorem \ref{necpp}.
If $a=1$, then $2=DD^{(-1)} = 2 + (\zeta_h^{a_0-a_1}+\zeta_h^{a_1-a_0})g$
and hence $\zeta_h^{a_0-a_1}+\zeta_h^{a_1-a_0}=0$. This implies $\zeta_h^{2a_0-2a_1}=-1$
and thus $h \equiv 0\ (\bmod\ 4)$. This shows that (\ref{necsuff1}) is necessary in the case $a=1$.
As $h$ is even, condition (\ref{necsuff1}) is also necessary in the case $a=2$.
Now suppose $a\in \{3,4\}$. To prove the necessity of (\ref{necsuff1}), it suffices
to show that $\nu_2(h)=1$ is impossible. Thus assume $\nu_2(h)=1$.
Similar to the case $p=2$ in the proof of Theorem \ref{necpp}, we see that we
can assume $\chi(D) \in \Z[\zeta_{2h}]$ where $\chi$ is the character of $G$ with $\chi(g)=\zeta_{2^a}$
(note that we have $\chi(D) \in \Z[\zeta_{2h}]$ and
not only $\chi(D) \in \Z[\zeta_{4h}]$, as  $h$ is even now).
Furthermore, as in the proof of Theorem \ref{necpp}, we have
$D=\sum_{j=0}^3 \zeta_h^{a_{0,j}}g^j$,
$A\bar{A}+B\bar{B} = 2^a$ where $A=\zeta_h^{a_{0,0}} - \zeta_h^{a_{0,2}}$
and $B= \zeta_h^{a_{0,1}} - \zeta_h^{a_{0,3}}$.
This implies $a=3$ and $|A|=|B|=2$.
This is only possible if $\zeta_h^{a_{0,2}}=-\zeta_h^{a_{0,0}}$
and $\zeta_h^{a_{0,1}}=-\zeta_h^{a_{0,3}}$.
But then $\chi_0(D) = \sum_{j=0}^3 \zeta_h^{a_{0,j}}=0$
for the trivial character $\chi_0$ of $G$, a contradiction.
In summary, we have shown that (\ref{necsuff1}) is necessary if $p=2$ and $a\le 4$.

\medskip

By what we have shown, we can assume $a\ge 5$ if $p=2$.
Suppose that condition (\ref{necsuff1})  is not satisfied.
Set $b=\nu_p(h)$ if $p$ is odd and $b=\max\{2,\nu_2(h)\}$ if $p=2$.
By Theorem \ref{necpp}, we have $b\ge 1$.
Moreover, $\nu_p(h)<a/2$, as we assume that  (\ref{necsuff1})  does not hold.
Note that, for $p=2$, this implies $b<a/2$, as $a\ge 5$ and thus $a/2>2$ is this case.
In summary, we have
\begin{equation} \label{necsuff1a}
\text{$1\le b<a/2$  if  $p$ is odd and $2\le b<a/2$ if $p=2$.}
\end{equation}

\medskip

Let $\chi$ be the character of $G$ with $\chi(g)=\zeta_{p^a}$.
Note that $|\chi(D)|^2=p^a$ by Lemma \ref{gr}. 
Write $h=p^bh'$ where $(h',p)=1$. 
By Result \ref{fdp}, there is an integer $s$ such
$\chi(D)\zeta_{p^a}^s\in \Z[\zeta_{ph'}]$ if $p$ is odd
and $\chi(D)\zeta_{p^a}^s\in \Z[\zeta_{4h'}]$ if $p=2$. 
In view of (\ref{necsuff1a}), this implies $\chi(D)\zeta_{p^a}^s\in \Z[\zeta_{h}]$.
Replacing $D$ by $Dg^{-s}$, if necessary,
we can assume
\begin{equation} \label{necsuff1b}
\chi(D)\in \Z[\zeta_h].
\end{equation}
Write $h=p^bm$, where $(p,m)=1$.
Every integer in $\{0,\ldots,p^a-1\}$ has a unique representation as $i+p^{a-b}j$ with $0\le i \le p^{a-b}-1$ and $0\le j \le p^b-1$.
Hence we can write
$$
D = \sum_{i=0}^{p^{a-b}-1}\sum_{j=0}^{p^b-1} \zeta_{p^b}^{a_{i,j}} \zeta_{m}^{b_{i,j}} g^{i+p^{a-b}j},
$$
with $a_{i,j}, b_{i,j} \in \Z$   and we have
$$
\chi(D) = \sum_{i=0}^{p^{a-b}-1}\zeta_{p^a}^i\sum_{j=0}^{p^b-1} \zeta_{p^b}^{a_{i,j}+j} \zeta_{m}^{b_{i,j}}.
$$
In view of (\ref{necsuff1b}), and since $\{1,\zeta_{p^a},\ldots,\zeta_{p^a}^{p^{a-b}-1}\}$ is linearly independent
over $\Q(\zeta_{p^bm})$, we have $\sum_{j=0}^{p^b-1} \zeta_{p^b}^{a_{i,j}+j} \zeta_{m}^{b_{i,j}}=0$
for all $i>0$. We conclude $\chi(D)=\sum_{j=0}^{p^b-1} \zeta_{p^b}^{a_{0,j}+j} \zeta_{m}^{b_{0,j}}$
and hence $|\chi(D)|^2 \le p^{2b}$.
Thus $|\chi(D)|^2<p^a$ by (\ref{necsuff1a}), contradicting $|\chi(D)|^2=p^a$.
\end{proof}

\bigskip

\noindent
{\bf Acknowledgement}
We thank the anonymous referees for their careful reading of 
the manuscript and their suggestions that helped
to improve the readability of the paper.


\end{document}